\documentclass{article}
\usepackage{color, amsmath, amstext, latexsym}
\usepackage{amsfonts}
\usepackage{amssymb}
\usepackage{amsthm}
\usepackage{dsfont}

\newcommand{\al}{{\alpha}}

\newcommand{\ga}{{\gamma}}

\newcommand{\om}{{\omega}}
\newcommand{\Om}{{\Omega}}

\newcommand{\la}{\lambda}

\newcommand{\R}{{\mathbb R}}
\newcommand{\C}{{\mathbb C}}

\newcommand{\realstwo}{\mathbb{R}^2}
\newcommand{\realsthree}{\mathbb{R}^3}

\newcommand{\cO}{\mathcal  O}
\newcommand{\cE}{\mathcal  E}
\newcommand{\sO}{\mathcal  O}

\newcommand{\pd}{{\partial}}
\newcommand{\hf}{\frac12}

\newtheorem{lemma}{Lemma}[section]
\newtheorem{theorem}[lemma]{Theorem}

\newtheorem{prop}[lemma]{Proposition}

\begin{document}

\title{
Remark on an elastic plate interacting with a gas in a semi-infinite tube:
periodic solutions}

\author{Igor
Chueshov\thanks{e-mail:
chueshov@karazin.ua}  
\\ \\
School of  Mathematics and Informatics, 
\\ Karazin Kharkov
National University, \\ Kharkov, 61022,  Ukraine
 }

\maketitle
\begin{abstract}
We consider a conservative system consisting of an elastic plate interacting
with  a gas filling a semi-infinite tube. 
The plate is placed on the bottom of the tube.
The dynamics of the gas velocity potential is governed by the linear wave equation.
The plate displacement satisfies the linear Kirchhoff equation. 
We show that this system possesses an infinite number of periodic solutions with the frequencies tending to infinity. This means 
 that the well-known property of  decaying of local wave energy in tube domains does not hold for the system considered.

\end{abstract}
\bigskip\par\noindent

Recently there was great interest in the study 
of long-time dynamics of  elastic plates interacting with a flow of gas
(see, e.g., \cite{CDLW1,CDLW2,cl-book,jadea12,supersonic,delay,fereisel,conequil1,conequil2} and the literature cited in those sources).
The corresponding model has the form
\begin{equation}\label{flowplate1}\begin{cases}
u_{tt}+\Delta^2u+\ga u_t +f(u)= \big(\partial_t+U\partial_x\big)\left[\phi \vert_{x_{3}=0}\right] ~~\text { in } \Omega\times \R_+,\\
u(0)=u_0;~~u_t(0)=u_1,\\
u=\Delta  u = 0 ~~ \text{ on } \partial\Omega\times \R_+,
\end{cases}
\end{equation}
where $\phi(x_1,x_2,x_3;t)$ 
solves the problem in
$\realsthree_+\equiv \{x=(x_1;x_2;x_3): x_3>0\}$:
\begin{equation}\label{flowplate2}\begin{cases}
(\partial_t+U\partial_x)^2\phi=\Delta \phi ~~ \text { in } \realsthree_+ \times \R_+,\\
\phi(0)=\phi_0;~~\phi_t(0)=\phi_1 ~~ \text { in } \realsthree_+\\
\pd_{x_3} \phi = 
\begin{cases}
(\partial_t+U\partial_x)u (x_1,x_2)& \text{ on } 
\Om  \times \R_+, \\
0  & \text{ on } 
(\R^2\setminus\overline{\Om})  \times \R_+.
\end{cases}
\end{cases}
\end{equation}
Here above 
 $\Om$ is a bounded smooth domain  
in $\realstwo$ identified with
\[
\{(x_1;x_2; 0): (x_1;x_2)\in\Om \}\subset\overline{\realsthree_+}.
\]  The term
$f(u)$ describes a nonlinear 
force which can be (a)
von Karman type (like as in \cite{cl-book}), or (b) Berger type (like as in \cite{Chu99}), or even (c) generated by some   Nemytskii operator (see, e.g., \cite{cl-kb1,cl-kb2}).  
 The unknown function $u = u(x_1,x_2;t)$  measures the transverse displacement of the plate
at the point $(x_1;x_2)$ and time $t$. The boundary conditions for $u$ means that
the plate is hinged on its edge.
The function  $\phi (x,t) = \phi ( x_{1},x_{2}, x_{3}; t )$  is velocity potential
 of the gas filing
the domain $\realsthree_+$. Here we deal with
 interaction of a plate with a  gas flow moving with the speed $U$ in the direction of the axis $x_1$.
The aerodynamical pressure of the gas  on the plate is given by the term
$p(x,t) =  \nu (\phi_t+U\phi_{x_1}) \vert_{ x_{3}=0}$,
the parameter $\nu > 0$
characterizes the intensity of the interaction between the gas and the
plate. The  transverse displacement $u(x,t)$  has influence on the
gas via boundary condition in (\ref{flowplate2}).
\par 

For recent  surveys of mathematical and applied aspects
of the model above we refer to
\cite{CDLW1,CDLW2}.
Here we only mention  the convergence results in  the subsonic case ($0\le U<1$) which were established in
\cite{conequil1,conequil2} (see also \cite{cl-book} for related facts) and state  
stabilization of solutions to stationary states 
of the system when $t\to  \infty$ under some conditions concerning initial data of the gas velocity potential. 
The corresponding argument requires positivity of the damping parameter $\ga$ and
involves a gradient-type structure of the system
in the case considered.
\par 
In this relation the question (see \cite{chuey}
and \cite[p.694]{cl-book}) arises whether it is possible to obtain  a similar stabilization result
in the absence ($\ga=0$) of the internal damping in the plate. This conjecture is based on the well-known property of 
 the local energy decay for the wave equation in
 $\realsthree$ and some other unbounded domains (see also Proposition~\ref{pr:wave-decay} below).
 
 Our main goal in  this note is to show that 
 unboundness of the wave domain $\cO$ is not sufficient to guarantee stabilization 
 of solutions to equilibria. For this we
consider a  linear plate model without any damping ($\gamma =0$), coupled to the flow via matching velocities. The parameter  $U$ is taken to be zero. We show that in this scenario,  periodic solutions may exist.  
In the case of a specific tubular domain
they are explicitly constructed. Whether the same result  holds for nonlinear  plate is an open question. However, the model indicates the necessity of introducing mechanical damping in the plate model {\em if one expects a strong convergence to equilibria} of the full flow-structure system. 
 \par 

\section{Model}
Let $\Omega $  be a smooth bounded domain in ${\R}^{2}$.
We consider the following problem
\begin{equation}\label{1}
\left\{
\begin{array}{l}
\partial ^{2}_{t}u  + \Delta ^{2}u -
\nu \cdot\partial _{t} \left[\phi \vert_{x_{3}=0}\right]=0,
 \quad x=(x_1;x_2) \in  \Omega,\; t>0, \\ \\
u\vert_{\partial \Omega } = \Delta u\vert_{\partial \Omega } = 0,
\quad  t>0,
\\ \\
u\vert_{t=0} = u_{0}(x), \quad \partial _t u\vert_{t=0} = u_{1}(x),
\quad x=(x_1;x_2) \in  \Omega ,
\end{array}
\right.
\end{equation}
We denote  by $\phi \vert_{x_{3}=0}$  the trace
of a function $\phi(x_1,x_2,x_3,t)$ in $\Omega\times \R_+\times {\R}_+$
which solves the problem
\begin{equation}\label{2}
\left\{
\begin{array}{l}
\partial _{tt} \phi  = \Delta\phi  , \quad
x \in  \sO_{+} \equiv
\{x= (x_{1};x_{2};x_{3})\;:\; (x_1;x_2)\in {\Om},\; x_{3} >0 \},
\\ \\
\partial_{x_{3}}\phi
=\partial _{t}u(x_{1},x_{2},t),\;
 (x_1;x_2)\in\Omega, \;x_3=0,\; t>0,\\ \\
 \phi=0,\; (x_1;x_2)\in\partial \Omega,\; x_3>0,\; t>0,
 \\ \\
\phi \vert_{t=0} = \phi _{0}(x), \quad
\partial _{t}\phi \vert_{t=0} =
\phi _{1}(x),\quad x \in  \sO_{+}.
\end{array}
\right.
\end{equation}
The function $u = u(x,t)$  is the transverse displacement of the hinged plate. 
The function  $\phi (x,t) = \phi ( x_{1},x_{2}, x_{3}; t )$  is velocity potential
 of the gas filing
the tube $\sO_+$. 
The pressure of the gas  on the plate is given by the term
$p(x,t) =  \nu \cdot \phi_t \vert_{ x_{3}=0}$,
where $\nu > 0$.

In the case of bounded domains $\sO$
 systems like (\ref{1}) and (\ref{2} was studied 
by many authors (see the discussion and the references in  \cite{cl-book}  and  \cite{cmbs}).
\par 
 Below we use the notation $H^s(D)$ for the Sobolev space of order $s$ on a domain $D$ in $\R^d$, $d=2,3$.
\par   
We start with the following assertion.

\begin{prop}[Well-posedness]\label{pr:wp}
	Assume that 
	\[
	\phi_0\in H^1(\cO_+),~~\phi_1\in L_2(\cO_+),~~u_0\in (H^2\cap H^1_0)(\Om),
~~u_1\in L_2(\Om).
	\]
	Then there exists a unique couple $\{u;\phi\}$
	of function 
	\[
	u\in C^1(\R_+; L_2(\Om))\cap 
	C(\R_+; (H^2\cap H^1_0)(\Om))
	\]
and
	\[
	\phi\in C^1(\R_+; L_2(\cO_+))\cap 
	C(\R_+; H^1(\cO_+))
	\]	
	solving  \eqref{1} and \eqref{2} in the sense of distributions. Moreover this solution satisfies the energy preservation law of the form
\[
\cE(t)\equiv E^{gas}_{\cO_+}(\phi_t(t),\phi(t))+
\nu^{-1}E^{plate}_{\Om}(u_t(t),u(t))=\cE(0),~~
\mbox{for all}~~ t>0,
\]	
where we use the notations 
\begin{equation}\label{loc-energy}
	E^{gas}_{D}(\phi_t,\phi)=\hf\int_D\left[
|\phi_t(x,t)|^2 +|\nabla\phi(x,t)|^2\right] dx
\end{equation}
for every $D\subseteq\cO_+\subset\R^3$ and
\[
E^{plate}_{\Om}(u_t,u)=
\hf\int_\Om\left[
|u_t(x,t)|^2 +|\Delta u(x,t)|^2\right] dx.
\]
\end{prop}

\begin{proof}
We can apply well-known general result reported
in \cite{cmbs}, see also 
\cite[Chapter 6]{cl-book}, where nonlinear versions of similar problems are discussed.
However we can also give more direct argument based 	on some symmetry of this linear problem and involving 
the variables  separation.
We sketch the corresponding argument below.
\par

Let $\{ e_{k} \}$  be an orthonormal basis in $L_2 ( \Omega  )$ consisting of
eigenvectors of the problem
\begin{equation}\label{Dir-spec}
   \Delta w +\lambda  w=0, \quad
w\vert_{\partial \Omega } = 0,
\end{equation}
and $0<\lambda_1\le\lambda_2\le ...$ the corresponding eigenvalues.
We are looking for solutions to problem (\ref{1}) and (\ref{2}) in
the following form
\begin{equation}\label{1.1}
u(x_1,x_2,t)=\sum_{k=1}^\infty u_k(t)e_k(x_1,x_2),
\quad
\phi(x_1,x_2,x_3,t)=\sum_{k=1}^\infty \phi_k(t, x_3)e_k(x_1,x_2).
\end{equation}
It is clear that the function $\phi_k(t,z)$
solves the problem
\begin{equation}\label{1.2}
\left\{
\begin{array}{l}
\partial _{tt} \phi_k - \partial_{zz}\phi_k +\lambda_k\phi_k=0 , \quad
z>0,\; t>0,
\\ \\
 \partial_{z}\phi_k =\dot u_k(t),\; \;z=0,\; t>0,
 \\ \\
\phi_k \vert_{t=0} = \phi _{0k}(z), \quad
\partial _{t}\phi_k \vert_{t=0} =
\phi _{1k}(z),\quad z>0,
\end{array}
\right.
\end{equation}
where $u_k(t)$ satisfies the equation
\begin{equation}\label{1.3}
\ddot u_k +\lambda_k^2 u_k -\nu \partial_t\phi_k(0,t)=0, \quad
u_k \vert_{t=0} = u_{0k}, \quad
\dot u_k \vert_{t=0} =
u _{1k}.
\end{equation}
It is easy to show that for each $k$ problem \eqref{1.2} and \eqref{1.3}
has a unique  solution $(\phi_k(t),u_k(t))$ 
for which we have the following energy balance relation
\[
E_k(t)
=E_k(s),~~~ t\ge s,
\]
where the energy $E_k$ of the $k$-mode  has the form
\begin{align*}
E_k(t) = &\hf \int_0^\infty\left[ |\pd_t\phi_k(t,z)|^2 + |\pd_z\phi_k(t,z)|^2
+\lambda_k |\phi_k(t,z)|^2\right] dz\\ &+ \frac{1}{2\nu}\left[ |\dot u_k|^2 +\lambda_k^2 |u_k|^2\right].
\end{align*}
This observations  allow us to obtain appropriate a priori estimates and conclude the proof by the standard compactness method.
\end{proof}

\section{Dynamics}
We start with the following assertion that shows
a local energy decay in the case when the bottom 
$\Om$ of the cylinder $\cO_+$ is {\em rigid}.
This means that we consider the wave dynamics only. This dynamics is described by the following equations  
\begin{equation}\label{lin-eq-2}
\left\{
\begin{array}{l}
\partial _{tt} \phi  = \Delta\phi  , \quad
x \in  \sO_{+} ,
\\ \\
\partial_{x_{3}}\phi
=0,\;
 (x_1;x_2)\in\Omega, \;x_3=0,\; t>0,\\ \\
 \phi=0,\; (x_1;x_2)\in\partial \Omega,\; x_3>0,\; t>0,
 \\ \\
\phi \vert_{t=0} = \phi _{0}(x), \quad
\partial _{t}\phi \vert_{t=0} =
\phi _{1}(x),\quad x \in  \sO_{+}.
\end{array}
\right.
\end{equation}

\begin{prop}[Local energy decay]\label{pr:wave-decay}
	Assume that 
	$\phi_0\in H^1(\cO_+)$ and $\phi_1\in L_2(\cO_+)$.
	Then
problem \eqref{lin-eq-2} has a unique variational solution $\phi$ which belongs to the class	
	\[
	\phi\in C^1(\R_+; L_2(\cO_+))\cap 
	C(\R_+; H^1(\cO_+)).
	\]	
	This  this solution satisfies the energy preservation law of the form
\[
 E^{gas}_{\cO_+}(\phi_t(t),\phi(t))
=E^{gas}_{\cO_+}(\phi_1,\phi_0),~~
\mbox{for all}~~ t>0.
\]	
Moreover, we have decaying of $\phi$ as $t\to\infty$ in the local energy norm, i.e.,
\begin{equation}\label{wave-decay}
\lim_{t\to +\infty} E^{gas}_{\cO_+^R}(\phi_t(t),
\phi(t))=0~~ \mbox{for every}~~ R>0,
\end{equation}
where $E^{gas}_{\cO_+^R}$ is given by \eqref{loc-energy}  with
\[
\sO_{+}^R \equiv
\{x= (x_{1};x_{2};x_{3})\;:\; (x_1;x_2)\in {\Om},\; 0< x_{3} <R \}.
\]
\end{prop}
\begin{proof}
The existence and uniqueness of solutions to \eqref{lin-eq-2} is obvious (we can use 
the same idea as in Proposition \ref{pr:wp}, for instance). It is also clear that the energy relation is satisfied. Thus we only need to establish the property in \eqref{wave-decay}.
\par
Extending the initial data as even functions 
in the variable $x_3$ on the whole $x_3$-axis  we can consider the wave equation  in
the domain  
\[
\sO=
\{x= (x_{1};x_{2};x_{3})\;:\; (x_1;x_2)\in {\Om},\; -\infty < x_{3} <+\infty \}.
\]
with the Dirichlet boundary conditions on $\pd\cO$. Now we can separate variables as above an
apply the same idea as in \cite{walker} to prove \eqref{wave-decay} for \textit{localized} initial data $\phi_0$ and $\phi_1$. In fact the article \cite{walker} contains exactly this statement for the case when $\Om=(0,\pi)\times (0,\pi)$.
The method suggested in \cite{walker}  relies on the presentation of the solution $\phi$ in the form
\[
\phi(x_1,x_2,x_3,t)=\sum_{m,n=1}^\infty \phi_{mn}(t, x_3)e_{mn}(x_1,x_2),
\]
where $e_{mn}(x_1,x_2)=2\pi^{-1}\sin ma_1\sin nx_2$ are solutions to the spectral problem \eqref{Dir-spec} for $\Om=(0,\pi)\times (0,\pi)$ and $\phi_{mn}(t,z)$ solves the equation
\[
\partial _{tt} \phi_{mn} - \partial_{zz}\phi_{mn} +(m^2+n^2)\phi_{mn}=0 , \quad
z\in \R  ,\; t>0,
\]
Establishing appropriate bounds for $\phi_{mn}$ (see \cite{walker}) one can
prove the desired result for $\Om=(0,\pi)\times (0,\pi)$.
  The calculations given in \cite{walker} can be easily extended to the case of general domains $\Om$.
\par 
Then using approximation procedure for initial data  and the energy relation we can obtain the result for every pair 
$( \phi_1;\phi_0)\in L_2(\cO_+)\times H^1(\cO_+)$.
\end{proof}

The  decay property of the local wave energy 
demonstrated in Proposition \ref{pr:wave-decay}
is not valid
for the coupled system in \eqref{1} and \eqref{2}.
More precisely, we show that problem   
\eqref{1} and \eqref{2} possesses infinite number of periodic solutions with different periods.

\begin{theorem}\label{th:periodic}
	Let $\{\la_k\}$ be the eigenvalues of problem 
	\eqref{Dir-spec} and $\{e_k\}$ be the corresponding eigen-basis.
Then there exist sequences $\{\om_k\}$ nad $\{\al_k\}$ of positive numbers with properties
\begin{equation}\label{lim-prop}
\lim_{k\to\infty}\big[\om^2_k-\la_k\big]=0
~~\mbox{and}~~
\lim_{k\to\infty}\al_k\la_k=\nu 	
\end{equation}
such that the functions
\begin{equation}\label{phi-k}
\varphi^k (x_1,x_2,x_3;t)= \big[ A_k \cos\om_kt+
B_k\sin\om_kt\big] e^{-\al_k x_3} e_k(x_1,x_2) 
\end{equation} and 
\begin{equation}\label{u-k}
	 u^k (x_1,x_2;t)=  -\frac{\al_k}{\om_k}\big[ A_k \sin\om_kt- B_k
\cos\om_kt\big]  e_k(x_1,x_2),	
\end{equation}
where $A_k$ and $B_k$ are arbitrary real numbers,
solve problem 
\eqref{1} and \eqref{2} with appropriate initial data. Each trajectory $(\varphi^k;\varphi_t^k;u^k;u_t^k)$ is Lyapunov stable in the phase space ${\mathcal H}= H^1(\cO_+)\times 	L_2(\cO_+)\times H^1(\Om)\times L_2(\Om)$.
\end{theorem}
\begin{proof}
Let us look for solutions to   \eqref{1.2} and \eqref{1.3} of the form
\[
\phi_k(t,z)=e^{ i\om t} e^{-\al z},~~~ u_k(t) =a e^{ i\om t}
\]
with $\al>0$ and $\om,a\in \C$.
The substitution in \eqref{1.2} and \eqref{1.3} gives us the relations
\begin{align*}
& -\om^2-\al^2+\la_k=0, \\
& \al=i\om a, \\
& a(-\om^2+\la^2_k)+ i\nu \om =0.
\end{align*}
This implies that $a=-i\al\om^{-1}$ and also
\begin{align}\label{om-alpha}
 \om^2 =\la_k-\al^2, ~~~
 \om^2 =\frac{\al}{\al+\nu} \la^2_k.
\end{align}
One can see for every $k$ there exists unique solution $(\om^2_k,\al_k)$ to \eqref{om-alpha}.
It is also easy to find  that 
\[ (\om^2_k,\al_k)\sim (\la_k, \nu\la_k^{-1})
~~\mbox{when}  ~~ k\to+\infty
\] 
 in the sense    of \eqref{lim-prop}.
This implies the structure of a solution  written in \eqref{phi-k} and \eqref{u-k}.
\par 
Stability properties of solutions  follow from the energy preservation law.
\end{proof}
Theorem \ref{th:periodic} shows that the elasticity of the bottom $\Om$ of the cylinder $\cO_+$ destroy the local energy decay property which we observe in the case of rigid bottom (see Proposition~\ref{pr:wave-decay}).
\par
\medskip \par
We conclude this note with several open questions which, we believe, are important for understanding of long-time dynamics of flow-structure systems.

\subsubsection*{Open Questions:}
\begin{itemize}
  \item 
  Can we show that the minimal subspace 
  in ${\mathcal H}$ containing all solutions
   $(\varphi^k;\varphi_t^k;u^k;u_t^k)$ is asymptotically  stable?
  Is this subspace  a global minimal attractor?
  If not, what is a real candidate on the role of global attractor for \eqref{1} and \eqref{2}? 
  \item What  can we  say about  stability and spectral properties of the generators of  $C_0$ semigroups
  generated by \eqref{1} and \eqref{2} and its dissipative perturbation?
  For instance, is it possible to stabilize the system by introducing internal damping in the plate component only?
\end{itemize}
These questions are important not only for linear  dynamics, but also
 for  nonlinear perturbations of \eqref{1} and \eqref{2}.

\end{document}